\documentclass{article}

\usepackage{amssymb}
\usepackage{amsmath}
\usepackage{amsthm}
\newtheorem{theorem}{Theorem}[section]
\newtheorem{lemma}[theorem]{Lemma}

\theoremstyle{definition}

\newtheorem{example}[theorem]{Example}

\theoremstyle{remark}

\numberwithin{equation}{section}

    \newcommand{\st}{^{\textstyle {\ast }}}
   \newcommand{\ph}{\mbox{$\varphi$}}
    
    \renewcommand{\phi}{\varphi}
   \newcommand{\Ht}{\mbox{$H^{2}$}}

   \newcommand{\D}{\mbox{$\mathbb{D}$}}
   \newcommand{\C}{\mbox{$C_{\varphi}$}}
   \newcommand{\Cs}{\mbox{$C_{\varphi}\st$}}

 	\newfont{\caps}{cmcsc9}  
 	\newfont{\jour}{cmti9}  
\begin{document}
\title{Binormal, Complex Symmetric Operators}
\author{Caleb Holleman, Thaddeus McClatchey, Derek Thompson}
\maketitle

\begin{abstract}
    In this paper, we describe necessary and sufficient conditions for a binormal or complex symmetric operator to have the other property. Along the way, we find connections to the Duggal and Aluthge transforms, and give further properties of binormal, complex symmetric operators. 
\end{abstract}

\section{Introduction} \label{intro}
    
    
    
    

    An operator $T$ on a Hilbert space is said to be a \textit{complex symmetric operator} if there exists a \textit{conjugation} $C$ (an isometric, antilinear involution) such that $CTC = T^*$. An operator is said to be \textit{binormal} if $T^*T$ and $TT^*$ commute. 
    Garcia initiated the study of complex symmetric operators \cite{GarciaPutinar}; researchers have made much progress in the last decade. The study of binormal operators was initiated by Campbell \cite{Campbell} in 1972. 
        
    A normal operator is both binormal and complex symmetric. As we will see below, there are many classes of operators for which normal operators are a subclass, several of which are equivalent to normality in finite dimensions. However, both binormality and complex symmetry are meaningful for matrices. Both are properties that automatically transfer to the adjoint of an operator. Furthermore, both have a joint connection to square roots of normal operators. For these reasons, we investigate the possible connections between binormality and complex symmetry.
    
    In the rest of Section \ref{intro}, we give necessary preliminaries. In Section \ref{cso to bin}, we give exact conditions for when a complex symmetric operator is binormal. In Section \ref{bin to cso}, we give conditions for when a binormal operator is complex symmetric in terms of the Duggal and Aluthge transforms. In Section \ref{bin and cso}, we discuss properties of binormal, complex symmetric operators. Lastly, in Section \ref{questions}, we pose questions for further study. 
    \newline
    
    \textbf{A possible point of confusion.} Campbell realizes in his second paper \cite{Campbell2} that Brown \cite{Brown} had already used the term binormal for an entirely different condition. Campbell stated that usage of the term was not current, so he continued with his use of binormal to describe the definition we use here. Unfortunately, both definitions are still in use \cite{GarciaWogen}, \cite{Kim}. In particular, Garcia and Wogen have shown that every binormal operator (in the sense of Brown) is complex symmetric \cite{GarciaWogen}. Throughout this paper, we use the term only in the sense of Campbell. 
    \newline

    \subsection{Preliminaries}
        An operator $T$ is
        
        \begin{enumerate}
            \item \textit{normal} if $T$ commutes with $T^*$;
            \item \textit{quasinormal} if $T$ commutes with $T^*T$;
            \item \textit{subnormal} if $T$ is the restriction of a normal operator to an invariant subspace;
            \item \textit{hyponormal} if $||Tx|| \geq ||T^*x|| \textrm{ }\forall x$.
        \end{enumerate}
        
        These properties form a chain of inheritance: normality implies quasinormality; quasinormality implies subnormality; and subnormality implies hyponormality. Every hyponormal operator with a spectrum of zero area is normal \cite{Stampfli}, so these four conditions are equivalent and therefore uninteresting in finite dimensions. The next result shows that complex symmetry stands entirely apart from these other properties when $T$ is not normal. The proof below is owed to Garcia and Hammond \cite{GarciaHammond}.

        \begin{theorem}\label{hyponormal} An operator which is both hyponormal and complex symmetric is normal.\end{theorem}
            \begin{proof}  Given a hyponormal operator  $T$, $||Tx|| \geq ||T^*x|| \textrm{ }\forall x$. Since $T$ is also complex symmetric, there is a conjugation $C$ so that $T = CT^*C$ and $CT = T^*C$. Then $$||Tx|| = ||CT^*Cx|| = ||T^*Cx|| \leq ||TCx|| = ||CT^*x|| = ||T^*x||.$$
                
            Since we now also have $||Tx|| \leq ||T^*x||$, $T$ is normal. \end{proof}

        Binormality is an offshoot in the above chain of inherited properties. Every quasinormal operator is binormal, but not every subnormal operator is binormal. An operator can be non-trivially complex symmetric and binormal, unlike above.

        There is a further connection between binormality and complex symmetry that shows their union lives in a different space than hyponormal operators. When $T^2$ is normal, $T$ is never non-trivially hyponormal \cite{Campbell2}, but we have a much different case here. The following result provides context for later results related to squares of operators. 
        
        \begin{theorem}\label{square} If $T^2$ is normal, then $T$ is both binormal and complex symmetric.  \end{theorem}
            \begin{proof} If $T^2$ is normal, then $T$ is binormal by \cite[Theorem 1]{CampbellEP}, and complex symmetric by \cite[Corollary 3]{GarciaWogen}. \end{proof}
        
        \begin{example}\label{compop} Consider the operator \C\ on \Ht\ induced by the involutive automorphism $\ph = \frac{a-z}{1-\overline{a}z}$ for some $a \in \D$. This operator was proved to be binormal in \cite{Kim} and complex symmetric in \cite{GarciaHammond}. Both properties quickly follow from the fact that $(\C)^2 = I$ and Theorem \ref{square}.  Furthermore, the only normal composition operators on \Ht\ are induced by functions of the form $|\lambda|z, |\lambda| \leq 1$. Therefore, \C\ is complex symmetric and binormal, but not normal. \end{example}
        
        Unfortunately, Theorem \ref{square} is not biconditional, as the next example shows. While the complex symmetry of $T$ necessarily carries to $T^2$, there are binormal, complex symmetric operators $T$ such that $T^2$ is not even binormal, let alone normal. 
        
        \begin{example} The matrix $T=\begin{bmatrix}
        -1 & 0 & -1\\ 
        -1 & 0 & 1\\ 
        0 & 1 & 0
        \end{bmatrix}$ is binormal, and complex symmetric by the Strong Angle Test \cite{Balayan}. However, $T^2=\begin{bmatrix}
        1 & -1 & 1\\ 
        1 & 1 & 1\\ 
        -1 & 0 & 1
        \end{bmatrix}$, which is not binormal.
        \end{example}
        
        Now, since we do not have an exact characterization of when binormality and complex symmetry jointly exist, we instead turn our focus in the next two sections to how one property can induce the other.

\section{When are complex symmetric operators binormal?} \label{cso to bin}

    We begin with the assumption that $T$ is complex symmetric. In this direction, we have an exact characterization of binormality, in terms of the conjugation $C$ corresponding to $T$. 
    
    \begin{theorem} A complex symmetric operator $T$ with conjugation $C$ is binormal if and only if $C$ commutes with $TT^*T^*T$ (equivalently, $T^*TTT^*$). \end{theorem}\label{commute}
    
        \begin{proof} Since $T$ is complex symmetric,  $CT = T^*C$ and $TC = CT^*$. In the first direction, we will assume that $C$ commutes with $TT^*T^*T$. Then 
        \begin{align*}
            T^*TTT^* &= T^*TCCT^*T \\
            &= T^*CT^*T^*CT^* \\
            &= CTT^*T^*TC \\
            &= TT^*T^*TCC \\
            &= TT^*T^*T, 
        \end{align*} so $T$ is binormal.
        
        Conversely, suppose $T$ is binormal. Then 
        \begin{align*} 
        TT^*T^*TC &= T^*TTT^*C \\ 
        &= T^*TCCTT^*C \\
        &= T^*CT^*T^*CT^*C \\ 
        &= CTT^*T^*TCC \\
        &= CTT^*T^*T.
        \end{align*} 
        In either direction, the proofs using $T^*TTT^*$ are identical. \end{proof}

    

    In general, if $T$ and $C$ are well understood, then Theorem \ref{commute} is not productive. One can simply write $T^* = CTC$ and check binormality directly. However, if $T$ is understood to be binormal and complex symmetric without knowing the appropriate conjugation (for example, if $T^2$ is normal), then Theorem \ref{commute} puts necessary conditions on possible choices for $C$. 

\section{When are binormal operators complex symmetric?} \label{bin to cso}

    In this section, we assume $T$ is binormal and find conditions for which $T$ is complex symmetric. We give an exact characterization in terms of the Duggal transform, and describe the stranger situation involving the Aluthge transform. First, we define the two operations. 
    \newline
    
    \textbf{Definition.} If $T = U|T|$ is the polar decomposition of $T$, then the \textit{Duggal transform} $\widehat{T}$ of $T$ is given by $\widehat{T} = |T|U$. 
\newline

 \textbf{Definition.} If $T = U|T|$ is the polar decomposition of $T$, then the \textit{Aluthge transform} $\tilde{T}$ of $T$ is given by $\tilde{T}=|T|^{1/2}U|T|^{1/2}$.
    
    \begin{theorem}\label{duggal} A binormal operator $T$ with polar decomposition $T = U|T|$ is complex symmetric if and only if $\widehat{T} = |T|U$ is complex symmetric. \end{theorem}
        \begin{proof} 
        Suppose $T$ is complex symmetric. Then, $\widehat{T}$ is also complex symmetric, according to \cite[Theorem 1]{GarciaAluthge}. 
        
        Suppose $\widehat{T}$ is complex symmetric for a conjugation $C$. Then, since $T$ is binormal, the polar decomposition of $\widehat{T}$ is $\widehat{T} = \widehat{U} |\widehat{T}|$ \cite[Theorem 3.24]{Saji}. Since $\widehat{T}$ is complex symmetric, $\widehat{U} = U^*UU$ is unitary \cite[Corollary 1]{GarciaPutinar2}. If $U$ is a proper partial isometry, then $U^*U$ is a proper projection and $U^*UU$ cannot be unitary. Therefore, $U$ is a full isometry, meaning that $U^*U = I$ and so $U^*UU = U$. This justifies that $U$ is unitary, and so $\widehat{T}U^{*} =|T|$. Then we have $T = U|T| = U\widehat{T}U^{*} $ and therefore $T$ and $\widehat{T}$ are unitarily equivalent. Therefore, $T$ is also complex symmetric. \end{proof}
        
    \begin{example} Recall the operator \C\ on \Ht\ induced by an involutive disk automorphism \ph, as in Example \ref{compop}. In \cite{Noor1}, it is established that the polar decomposition of \C\ is given by $\C = U|\C|$, where $U = \Cs |\C|$ is self-adjoint and unitary. By Theorem \ref{duggal}, $\widehat{\C} = |\C| \Cs |\C|$ is complex symmetric. This can also be seen by the fact that $\widehat{\C}$ is again involutive.    \end{example}

    Theorem \ref{duggal} is a nice characterization, but it proves something much stronger than the preservation of complex symmetry. If $T$ is binormal and either $T$ or $\widehat{T}$ is complex symmetric, then $T$ and $\widehat{T}$ are unitarily equivalent. Therefore, $\widehat{T}$ may be too similar to $T$ to be useful. Instead, we now turn to the Aluthge transform. In \cite{GarciaAluthge}, we find that $\tilde{T}$ and $T^2$ are connected by the fact the kernel of the Aluthge transform is exactly the set of operators which are nilpotent of order two. The connection between $\tilde{T}$ and $T^2$ is highlighted again in our next theorem. Our following proofs still involve the Duggal transform and the following preliminary.
    
    \begin{lemma} \label{TTcommutation} An operator $T = U|T|$, with $U$ unitary, is binormal if and only if $|T|$ and $|\widehat{T}|$ commute. \end{lemma}
        \begin{proof}
        Suppose $T$ such that $[|\widehat{T}|,|T|]=0$. By the proof of \cite[Theorem 3.2.9]{Saji}, $T$ is binormal if and only if $\widehat{T}$ is binormal, or, equivalently, $[|\widehat{T}|,|\widehat{T}^*|]=0$. Now, $|\widehat{T}^*|=|(|T|U)^*|=|U^*|T||=\left(|T|UU^*|T| \right )^{1/2}=|T|$. Therefore, $[|\widehat{T}|,|\widehat{T}^*|]=[|\widehat{T}|,|T|]=0$, and so $T$ is binormal.
        
        Suppose $T$ is binormal. As before, by \cite{Saji}, $\widehat{T}$ is binormal: $[|\widehat{T}|,|\widehat{T}^*|]=0$. Since $|\widehat{T}^*|=|T|$, $0=[|\widehat{T}|,|\widehat{T}^*|]=[|\widehat{T}|,|T|]$, as desired.
        \end{proof}
    
    \begin{theorem}\label{binsquarecso} Suppose $T$ is binormal. If $\tilde{T}$ is complex symmetric, then $T^2$ is complex symmetric. \end{theorem}
    
        \begin{proof} Let $T$ be binormal and $\tilde{T}$ be a complex symmetric operator, and let $T=U|T|$ be the polar decomposition of $T$. Our first step is to show that $U$ is unitary; we proceed as in Theorem \ref{duggal}. By \cite[Theorem 3.2.4]{Saji}, $\tilde{T}=U^*UU|\tilde{T}|$ is the polar decomposition of $\tilde{T}$. Since $\tilde{T}$ is complex symmetric, we conclude by the same reasoning as before that $U^*UU = U$ and $U$ is unitary. Therefore, $\tilde{T}=U|\tilde{T}|$, so $U^{*}\tilde{T} = |\tilde{T}|$. Then:
        \begin{align*}
            U^*\tilde{T}U^*\tilde{T} &= | \tilde{T}|^2 \\
            &= \tilde{T}^* \tilde{T} \\
            &= |T|^{1/2}U^*|T|U|T|^{1/2} \\
            &= |T|^{1/2}U^* \widehat{T}|T|^{1/2} \\
            &= |T|^{1/2}| \widehat{T}||T|^{1/2} \\
            &= | \widehat{T}||T|,
        \end{align*}
        
        where the last step is performed via Lemma \ref{TTcommutation}. Continuing:
        \begin{align*}
            U^*\tilde{T}U^*\tilde{T}&=| \widehat{T}||T| \\
            &=U^*\widehat{T}U^*T \\
            &=U^*|T|T \\
            &=U^{*2}T^2.
        \end{align*}
        Finally, letting $U^2$ act on the left side of the first and last expressions, we find $U\tilde{T}U^*\tilde{T}=T^2$.
        
        Since $\tilde{T}$ is complex symmetric, for some conjugation $C$, $C\tilde{T}C=\tilde{T}^*$. Further, for some conjugation $J$, $U=CJ$ \cite[Theorem 2]{GarciaPutinar2}. Then, $U\tilde{T}U^*\tilde{T}=CJ\tilde{T}JC\tilde{T}$, and 
        \begin{align*}
            (U\tilde{T}U^*\tilde{T})^* &=\tilde{T}^*U\tilde{T}^*U^* \\ 
            &=\tilde{T}^*CJ\tilde{T}^*JC \\
            &=\tilde{T}^*CJ\tilde{T}^*JC \\
            &=C\tilde{T}JC\tilde{T}CJC.
        \end{align*}
        
        Letting $CJC$ act on both sides of $U\tilde{T}U^*\tilde{T}$: 
        \begin{align}
            CJC(U\tilde{T}U^*\tilde{T})CJC&=CJC(CJ\tilde{T}JC\tilde{T})CJC \nonumber\\
            &=C\tilde{T}JC\tilde{T}CJC \nonumber\\
            &=(U\tilde{T}U^*\tilde{T})^*. \nonumber
        \end{align}
        Thus, $U\tilde{T}U^*\tilde{T} = T^2$ is complex symmetric. Therefore, the complex symmetry of $\tilde{T}$ implies the complex symmetry of $T^2$.
        \end{proof}
    
    Note the following about the proof of Theorem \ref{binsquarecso}:
    \begin{enumerate}
        \item Theorem \ref{binsquarecso} is not biconditional. There are binormal operators with complex symmetric squares, but whose Aluthge transforms are not complex symmetric.
        \item Neither the complex symmetry of an operator's square nor of its Aluthge transform guarantee complex symmetry for the original operator.
        \item In \cite[Example 3.2.7]{Saji}, Saji exhibits a binormal operator where $\tilde{T}$ is not binormal. It is tempting to think that since $\tilde{T}$ preserves complex symmetry, it might also preserve binormality. Unfortunately, this is not the case--even if $T$ possesses both complex symmetry and binormality.
        \item Theorem \ref{binsquarecso} is false without the condition of binormality.
    \end{enumerate}
    These facts are demonstrated in the following examples.
    
    \begin{example} The operator $T=\begin{bmatrix}2 & 2 & -2 & 0\\ 0 & 0 & 0 & -1\\ 2 & -2 & -2 & 0\\ 1 & 0 & 1 & 0 \end{bmatrix}$ is binormal, and its square is complex symmetric, but $\tilde{T}$ is not--all of these by by the Modulus Test \cite{GarciaPoore}.
    \end{example}
    \begin{example} The operator $T=\begin{bmatrix} -2 & -1 & 2 & 2\\ 1 & 0 & 0 & 2\\ 0 & -2 & 2 & -1\\ 0 & -2 & -1 & 0 \end{bmatrix}$ is binormal and not complex symmetric, but both $T^2$ and $\tilde{T}$ are complex symmetric--all of these by the Modulus Test.
    \end{example}

    \begin{example}\label{uncentered} 
    The matrix $T=\begin{bmatrix}
    0 & 1 & 1\\ 
    0 & 1 & -1\\ 
    1 & 0 & 0
    \end{bmatrix}$ is binormal, and complex symmetric by the Strong Angle Test. However, $\tilde{T}$ is not binormal.
    \end{example}
    
    \begin{example} 
    The matrix $T=\begin{bmatrix}
    -1 & -1 & -1\\ 
    0 & -1 & -1\\ 
    1 & -1 & -1
    \end{bmatrix}$ is not binormal and $\tilde{T}$ is complex symmetric, but $T^2$ is not complex symmetric--all by the Strong Angle Test.
    \end{example}

\section{Properties of binormal, complex symmetric operators} \label{bin and cso}

In this section, we determine a few properties of binormal, complex symmetric operators. First, we characterize such finite-dimensional matrices  with distinct singular values. We then give results on a variety of other operator properties as they relate to binormality and complex symmetry.

\subsection{Matrices}
    
    \begin{theorem}  A binormal complex symmetric $n \times n$ matrix $T$ with $n$ distinct singular values is an involutive weighted permutation. \end{theorem}
    \begin{proof}
        By Campbell, $T$ is a weighted permutation \cite[Theorem 9]{Campbell2}. Suppose $T$ is not involutive. Since $T$ is binormal, $T^*T$ and $TT^*$ share the same set of eigenvectors. Since $T$ has distinct singular values, $TT^*$ and $T^*T$ have the same set of eigenvalues as well. However, the eigenvector/eigenvalue pairing of $TT^*$ is potentially rearranged in $T^*T$. Further, all eigenvectors of $TT^*$ and $T^*T$ are mutually orthogonal since both operators are Hermitian. Let $\{ U_i \}$ and $\{ V_i \}$ be the ordered sets of eigenvectors of $TT^*$ and $T^*T$ respectively, so that for eigenvalue $\lambda$ there is an index $i$ such that $TT^*U_i=\lambda U_i$ and $T^*TV_i=\lambda V_i$.
        
        Because the eigenvectors are orthogonal, we can write $\left\langle U_i,V_j\right\rangle =\delta _{i,\Lambda(j) }$ where $\delta_{a,b}$ is the usual Kronecker delta: $1$ if its arguments are identical, $0$ otherwise; and $\Lambda$ is the map which takes the indices of $V$ to the indices of $U$. That is, $U_i=V_{\Lambda(i)}$ or $TT^*V_{\Lambda(i)}=T^*TV_i$.
        
        Let $\{ q_i \}_{i=1}^{n} $ be the indices of the longest disjoint cycle in the permutation $\Lambda$ with length $n$, so that $U_{q_1}=V_{q_n},\ U_{q_2}=V_{q_1},\cdots \ U_{q_n}=V_{q_{n-1}}$. Then, $\Lambda(q_i)=q_{i\%(n+1)}$. 
        
        If $n=1$ or $n=2$, then $\lambda$ is involutive, and so is $T$, which is a contradiction.
        
        If $n\geq 3$ then consider $\left\langle U_{q_1},V_{q_2}\right\rangle$ and $\left\langle U_{q_2},V_{q_1}\right\rangle$. Then, $\left\langle U_{q_1},V_{q_2}\right\rangle =\delta _{q_1,\Lambda(q_2)}=\delta _{q_1,q_3}=0$, but $\left\langle U_{q_2},V_{q_1}\right\rangle =\delta _{q_2,\Lambda(q_1)}=\delta _{q_2,q_2}=1$. Since the inner products of $U$ and $V$ are not equal under a change of indices, the Modulus Test gives that T is not a complex symmetric operator which is a contradiction. Thus, $T$ is involutive.
    \end{proof}
\subsection{Paranormal operators}    
    In \cite{Campbell2}, Campbell remarks that a hyponormal operator is normal if and only if its square is normal. We can now weaken that hypothesis from hyponormality to paranormality, using the lemmas below.
    
    \begin{lemma}\label{campbell} A binormal operator is hyponormal if and only if it is paranormal. \end{lemma}
    
    \begin{proof} The proof is found in \cite[Theorem 4]{Campbell2}. \end{proof}
    
    \begin{lemma}\label{paranormal} A binormal, complex symmetric operator $T$ is normal if and only if it is paranormal. \end{lemma}
    
    \begin{proof} If $T$ is paranormal and binormal, it is hyponormal by Lemma \ref{campbell}. If $T$ is hyponormal and complex symmetric, it is normal by Theorem \ref{hyponormal}. The other direction is trivial. \end{proof}
    
    \begin{theorem} A paranormal operator $T$ is normal if and only if $T^2$ is normal. \end{theorem}
    
    \begin{proof} Suppose $T^2$ is normal. By Theorem \ref{square}, $T$ is binormal and complex symmetric. Since $T$ is paranormal, it is hyponormal by Lemma \ref{campbell}. Then, $T$ is normal by Lemma \ref{paranormal}. The other direction is trivial. \end{proof}
\subsection{Square-positive semi-definite operators}    
    A point made to initiate the study of binormal complex symmetric operators is the connection to square-normal operators, as Theorem \ref{binsquarecso} notes. In this section, we use binormality and complex symmetry to get an exact characterization of square-positive semi-definite operators. (Recall $T$ is \textit{positive semi-definite} if $\langle Tx, x\rangle \geq 0 \textrm{ }\forall x$, or equivalently, $T = S^*S$ for some $S$.)
    
    To get there, we begin with the following lemma about binormal, complex symmetric operators in general.
    
    \begin{lemma}\label{u2} Suppose $T=U|T|$ is binormal and complex symmetric. Then the polar decomposition of $T^2$ is $T^2 = U^2|T^2|$, and $|T^2| = |T||\widehat{T}| = |\widehat{T}||T|$. \end{lemma}
    \begin{proof} Since $T$ is complex symmetric, $U$ is unitary. Recall from Section \ref{bin to cso} that for a binormal, complex symmetric operator, the Duggal transform $\widehat{T} = |T|U$ has the polar decomposition $\widehat{T} = U|\widehat{T}|$. Then 
    $$T^2 = TT = U|T|U|T| = U\widehat{T}|T| = U^2 |\widehat{T}||T|$$ And by Lemma \ref{TTcommutation}, $|\widehat{T}|$ and $|T|$ commute, and the product of positive semi-definite operators is positive semi-definite (and $U^2$ is unitary since $U$ is). Thus we have written $T^2$ as a unitary operator and a positive semi-definite operator, and since the polar decomposition is unique, we have $U^2$ as the unitary piece and $|\widehat{T}||T| = |T||\widehat{T}| = |T^2|$. \\ \end{proof} 
    
    \begin{theorem} For a bounded operator $T = U|T|$ on a Hilbert space, the following are equivalent. 
    \begin{enumerate}
        \item $T^2$ is positive semi-definite;
        \item $T$ is binormal and complex symmetric, and $U$ is self-adjoint. 
    \end{enumerate}
    \end{theorem}
    \begin{proof} Suppose $T$ is binormal, and complex symmetric (so that $U$ is unitary), and $U$ is self-adjoint (so that $U^2 = I$). By the Lemma \ref{u2}, $T^2 = U^2|T^2| = |T^2|$. Then $T^2$ is positive semi-definite, since it is the unique positive root of the positive semi-definite operator $T^*T^*TT$. \\
    In the other direction, suppose $T^2$ is positive semi-definite. Then $T^2$ is self-adjoint, and therefore normal. Then by Theorem \ref{binsquarecso}, $T$ is complex symmetric and binormal. Then by Lemma \ref{u2}, $T^2 = U^2|T^2|$. But since $T^2$ is positive semi-definite, we have $T^2 = |T^2|$. Therefore, since the polar decomposition is unique, we have $U^2 = I$. 

    \end{proof}

\section{Further Questions} \label{questions}
    
    \begin{enumerate}
        \item Is there a better exact characterization of when a binormal operator is complex symmetric?
        \item Is there a nice, exact characterization of binormal, complex symmetric operators? 
        \item Does every binormal, complex symmetric operator have an invariant subspace?
        \item What can be said, in general, about the spectrum of binormal, complex symmetric operators?
        \item An operator is called \textit{centered} if the doubly-infinite sequence $$\{ \dots, (T^2)^*T^2, T^*T, TT^*, T^2(T^2)^*, \dots  \}$$ is a set of mutually commuting operators. Every centered operator is clearly binormal. By \cite[Theorem 3.2.12]{Saji}, an operator is centered if and only if every iteration of the Aluthge transform is binormal. Therefore, if $T$ is centered and complex symmetric, so is $\tilde{T}$. Example \ref{compop} is centered and complex symmetric, as is any other involution. What stronger theorems can be made if $T$ is assumed to be centered and complex symmetric, rather than binormal and complex symmetric?
    \end{enumerate}
    
\section{Acknowledgements}

The authors would like to thank the Women's Giving Circle at Taylor University for funding this research.

\footnotesize

\bibliographystyle{amsplain}

\end{document}